\documentclass[a4paper,reqno,10pt]{amsart}
\makeatletter
\newcommand*{\rom}[1]{\expandafter\@slowromancap\romannumeral #1@}
\makeatother
\usepackage{epsf,graphicx,epsfig}
\usepackage{amscd}
\usepackage{amsmath,latexsym,amssymb,amsthm}
\usepackage[nospace,noadjust]{cite}
\usepackage{textcomp}
\usepackage{setspace,cite}
\usepackage{mathrsfs,amsmath}
\usepackage{chemarrow}
\usepackage{lscape,fancyhdr,fancybox}
\usepackage{stmaryrd}
\usepackage[all,cmtip]{xy}
\usepackage{tikz-cd}
\usepackage{cancel}
\usepackage[usestackEOL]{stackengine}
\usetikzlibrary{shapes,arrows,decorations.markings}
\usepackage{graphicx}

\usepackage{color}
\usepackage{bbm, dsfont}
\usepackage{xcolor}
\usepackage{url}
\usepackage{enumerate}
\usepackage[mathscr]{euscript}

  \theoremstyle{plain}
\swapnumbers
    \newtheorem{thm}{Theorem}[section]
    \newtheorem{proposition}[thm]{Proposition}
     \newtheorem{theorem}{Theorem}
   \newtheorem{lemma}[thm]{Lemma}

    \newtheorem{subsec}[thm]{}
\theoremstyle{definition}
    \newtheorem{definition}[thm]{Definition}
        \newtheorem{remark}[thm]{Remark}
\theoremstyle{remark}

\title{}
\author{}
\date{}
\usepackage{amssymb}

\usepackage{hyperref}
\hypersetup{
colorlinks,
citecolor=blue,
filecolor=black,
linkcolor=blue,
urlcolor=black
}

\begin{document}

\title[Non-abelian cohomology of Lie $H$-pseudoalgebras]{Non-abelian cohomology of Lie $H$-pseudoalgebras and inducibility of automorphisms}

\author{Apurba Das}
\address{Department of Mathematics,
Indian Institute of Technology, Kharagpur 721302, West Bengal, India.}
\email{apurbadas348@gmail.com, apurbadas348@maths.iitkgp.ac.in}

\maketitle

\begin{abstract}
  The notion of Lie $H$-pseudoalgebra is a higher-dimensional analogue of Lie conformal algebras. In this paper, we classify the equivalence classes of non-abelian extensions of a Lie $H$-pseudoalgebra $L$ by another Lie $H$-pseudoalgebra $M$ in terms of the non-abelian cohomology group $H^2_{nab} (L, M)$. We also show that the group $H^2_{nab} (L, M)$ can be realized as the Deligne groupoid of a suitable differential graded Lie algebra. Finally, we consider the inducibility of a pair of Lie $H$-pseudoalgebra automorphisms in a given non-abelian extension. We show that the corresponding obstruction can be realized as the image of a suitable Wells map in the context.
\end{abstract}

\medskip

\medskip

\medskip

\begin{center}

\end{center}

\tableofcontents

\section{Introduction}\label{sec1}

\subsection{Lie $H$-pseudoalgebras} The notion of Lie $H$-pseudoalgebra over a cocommutative Hopf algebra $H$ was introduced by Bakalov, D'Andrea and Kac \cite{bakalov-andrea-kac} as a multivariable generalization of the concept of Lie conformal algebra \cite{dk,k}. A Lie $H$-pseudoalgebra is a left $H$-module $L$ equipped with a $H^{\otimes 2}$-linear map $[\cdot * \cdot]_L: L \otimes L \rightarrow H^{\otimes 2} \otimes_H L$, called the pseudobracket, satisfying skew-symmetry and the Jacobi identity (see Definition \ref{defn-lieh} for the detail explanation). Given a Lie $H$-pseudoalgebra $L$, one can associate a Lie algebra $\mathcal{L} = H^* \otimes_H L$ called the annihilation algebra of $L$. When $L$ is finitely generated as an $H$-module, and $H$ is Noetherian then the annihilation algebra $\mathcal{L}$ is a linearly compact Lie algebra. Moreover, there is a one-to-one correspondence between Lie $H$-pseudoalgebra representations of $L$ and discrete continuous representations of $\mathcal{L}$ satisfying a technical condition. Lie $H$-pseudoalgebras are also related to differential Lie algebras of Ritt and Hamiltonian formalism of the theory of nonlinear evolution equations \cite{ritt}. In \cite{bakalov-andrea-kac} the authors have introduced the cohomology of a Lie $H$-pseudoalgebra with coefficients in a representation and observed that the second cohomology group classifies the set of all equivalence classes of abelian extensions. See also \cite{wu} for the description of the cohomology in terms of Maurer-Cartan elements in a suitable graded Lie algebra. Later, various representation theoretic studies of simple Lie $H$-pseudoalgebras were made in \cite{bakalov-andrea-kac2,bakalov-andrea-kac3,andrea-marchei}. See also \cite{liberati,retakh,wu2} for some other types of pseudoalgebras and associated theories.

\subsection{Non-abelian extensions and cohomology} The non-abelian cohomology was first considered by Eilenberg and Maclane to study group extensions with non-abelian kernel \cite{em} (see also \cite{hoch-serre}). Non-abelian cohomology generalizes the classical group cohomology (with coefficients in a module) that describes abelian extensions. Subsequently, non-abelian extensions and non-abelian cohomology for Lie algebras were studied by Hochschild \cite{hoch}. The coefficient in the non-abelian cohomology is not a representation space but another Lie algebra with a map that is a representation up to an inner automorphism. Please see \cite{neeb,fial-pen,fregier} for recent developments on non-abelian extensions of Lie groups, Lie algebras and some other types of algebraic structures. In particular, Fr\'{e}gier \cite{fregier} showed that the non-abelian cohomology group of a Lie algebra with values in another Lie algebra can be seen as the Deligne groupoid of a suitable differential graded Lie algebra. Further, the classification of abelian extensions of a Lie algebra by the Chevalley-Eilenberg cohomology can be realized as the tangent cohomology complex of the Deligne groupoid.

\subsection{The inducibility problem and the Wells exact sequence} Given an extension of some algebraic structure, the problem of inducibility of a pair of automorphisms is another interesting problem of research. This question was first considered by Wells \cite{wells} for abelian extensions of abstract groups and further studied in \cite{passi,jin-liu}. In particular, Wells constructed a map (now popularly known as the Wells map) and showed that the obstruction for the inducibility of a pair of group automorphisms lies in the image of that map. He also finds a short exact sequence connecting the Wells map and various automorphism groups. Given an abelian extension of Lie algebras, the inducibility problem and the Wells map were studied in \cite{bardakov}. However, the problem becomes more intricate when we start with non-abelian extensions of Lie algebras. This is simply because the domain and codomain of the Wells map will change. This problem can be stated as follows. Let $0 \rightarrow M \xrightarrow{i} E \xrightarrow{p} L \rightarrow 0$ be a given non-abelian extension of Lie algebras. For any $\gamma \in \mathrm{Aut}(E)$ with $\gamma (M) \subset M$, there is a pair of Lie algebra automorphisms $(\gamma|_M, p \gamma s)\in \mathrm{Aut}(M) \times \mathrm{Aut} (L)$, where $s$ is any section of the map $p$. A pair of Lie algebra automorphisms $(\beta, \alpha) \in \mathrm{Aut}(M) \times \mathrm{Aut} (L)$ is said to be inducible if there exists an automorphism $\gamma \in \mathrm{Aut}(E)$ with $\gamma (M) \subset M$ such that $(\gamma|_M, p \gamma s) = (\beta, \alpha)$. The inducibility problem then asks to find a necessary and sufficient condition for the inducibility of a pair of Lie algebra automorphisms.

\subsection{Layout of the present paper} Our aim in this paper is twofold. In the first part, we consider non-abelian extensions of Lie $H$-pseudoalgebras and also introduce the non-abelian cohomology $H^2_{nab} (L, M)$ of a Lie $H$-pseudoalgebra $L$ with values in another Lie $H$-pseudoalgebra $M$. We show that the set of all equivalence classes of non-abelian extensions of $L$ by $M$ can be classified by the cohomology group $H^2_{nab} (L,M)$. This result generalizes the classification result of abelian extensions \cite{bakalov-andrea-kac}. Next, given two Lie $H$-pseudoalgebras $L$ and $M$, we first construct a suitable differential graded Lie algebra $\mathfrak{g}$ using a generalization of the Nijenhuis-Richardson bracket (see \cite{wu} for the bracket). Finally, we show that the non-abelian cohomology $H^2_{nab} (L, M)$ can be realized as the Deligne groupoid of the differential graded Lie algebra $\mathfrak{g}$ constructed above. This result generalizes the work of Fr\'{e}gier in the context of pseudoalgebras. However, our proof is new and much simpler than \cite{fregier}. We also discuss abelian extensions in terms of the tangent cohomology complex. 

In the next part of the paper, we consider the inducibility of a pair of Lie $H$-pseudoalgebra automorphisms in a given extension. Unlike the Lie algebra case developed in \cite{bardakov}, we do not restrict ourselves to only abelian extensions. More precisely, we start with a non-abelian extension
 \begin{align*}
     \xymatrix{
     0 \ar[r] & (M, [\cdot * \cdot]_M) \ar[r]^i & (E, [\cdot * \cdot]_E) \ar[r]^p & (L, [\cdot * \cdot]_L) \ar[r] & 0
     }
     \end{align*}
of Lie $H$-pseudoalgebras and construct the Wells map $\mathcal{W} : \mathrm{Aut} (M) \times \mathrm{Aut}(L) \rightarrow H^2_{nab} (L, M)$ in the context. We show that a pair of Lie $H$-pseudoalgebra automorphisms $(\beta, \alpha) \in \mathrm{Aut} (M) \times \mathrm{Aut}(L)$ is inducible if and only if $\mathcal{W} ((\beta, \alpha)) = 0$. It shows that the image $\mathcal{W} ((\beta, \alpha))$ of the Wells maps is an obstruction for the inducibility of the pair $(\beta, \alpha)$ of Lie $H$-pseudoalgebra automorphisms. Finally, we derive a short exact sequence (a generalization of the Wells short exact sequence) connecting various automorphism groups and the Wells map.

\medskip

The paper is organized as follows. In Section \ref{sec2}, we recall some necessary background about Lie $H$-pseudoalgebras, their representations and cohomology. In Section \ref{sec3}, we study non-abelian extensions and the non-abelian cohomology of Lie $H$-pseudoalgebras. We interpret the non-abelian cohomology as the Deligne groupoid in Section \ref{sec4}. Finally, we consider the inducibility problem and construct the Wells short exact sequence in Section \ref{sec5}.

\section{Some background on Lie {\em H}-pseudoalgebras}\label{sec2}

In this section, we recall some necessary background on Lie $H$-pseudoalgebras. More precisely, we recall representations and cohomology of Lie $H$-pseudoalgebras following the references \cite{bakalov-andrea-kac,wu}.

Let $H$ be a cocommutative Hopf algebra with the comultiplication map $\Delta : H \rightarrow H^{\otimes 2}$ and the counit map $\varepsilon: H \rightarrow {\bf k}$. We follow Sweedler's notation to represent the images of the comultiplication map.

\medskip

Let $L$ be a left $H$-module. A ${\bf k}$-linear map $[\cdot * \cdot ]_L : L \otimes L \rightarrow H^{\otimes 2} \otimes_H L$ is said to be

(i) {\em $H^{\otimes 2}$-linear} if $[fx * gy]_L = ((f \otimes g) \otimes_H 1) [x * y]_L$, for all $x, y \in L$ and $f, g \in H$. Explicitly,
\begin{align*}
    \text{ if } [x * y]_L = \sum_{i} (f_i \otimes g_i) \otimes_H e_i ~ \text{ then } ~ [fx * gy]_L = \sum_{i} (ff_i \otimes gg_i) \otimes_H e_i.
\end{align*}

(ii) {\em skew-symmetric} if $[y * x]_L = - (\sigma_{12} \otimes_H 1) [x * y]_L$, for all $x, y \in L$. Here $\sigma_{12} : H^{\otimes 2} \rightarrow H^{\otimes 2}$ is the flip map. The skew-symmetric condition simply means that
\begin{align*}
    \text{ if } [x * y]_L = \sum_{i} (f_i \otimes g_i) \otimes_H e_i ~ \text{ then } ~ [y * x]_L = - \sum_{i} (g_i \otimes f_i) \otimes_H e_i.
    \end{align*}

Let $[ \cdot * \cdot]_L : L \otimes L \rightarrow H^{\otimes 2} \otimes_H L$ be a $H^{\otimes 2}$-linear map. Then one can define two $H^{\otimes 3}$-linear maps 
\begin{align*}
    [[ \cdot * \cdot]_L * \cdot]_L : L \otimes L \otimes L \rightarrow H^{\otimes 3} \otimes_H L ~~~\text{ and } ~~~  [\cdot * [ \cdot * \cdot]_L]_L : L \otimes L \otimes L \rightarrow H^{\otimes 3} \otimes_H L
\end{align*}
    as follows. Let $x, y, z \in L$. If 
$[x * y]_L = \sum_{i} (f_i \otimes g_i) \otimes_H e_i$  and say $[e_i * z]_L = \sum_{j} (f_{ij} \otimes g_{ij}) \otimes_H e_{ij},$
then $[[x * y]_L * z]_L$ is the element of $H^{\otimes 3} \otimes_H L$ given by
\begin{align}\label{xy-z}
    [[x * y]_L * z]_L = \sum_{i, j} \big( f_i f_{ij (1)} \otimes g_i f_{ij (2)} \otimes g_{ij}   \big) \otimes_H e_{ij}.
\end{align}
Similarly, if $[y * z]_L = \sum_i (h_i \otimes l_i) \otimes_H d_i$ and say $[x * d_i]_L = \sum_j (h_{ij} \otimes l_{ij}) \otimes_H d_{ij}$, then
\begin{align}\label{x-yz}
    [x * [y * z]_L]_L = \sum_{i, j} \big( h_{ij} \otimes h_i l_{ij (1)} \otimes l_i l_{ij (2)} \big) \otimes_H d_{ij} ~ \text{ as an element of } H^{\otimes 3} \otimes_H L.
\end{align}

\begin{definition}\label{defn-lieh}
    A {\em Lie $H$-pseudoalgebra} (or simply a {\em Lie pseudoalgebra} if $H$ is clear from the context) is a left $H$-module $L$ equipped with a $H^{\otimes 2}$-linear map $[\cdot * \cdot]_L : L \otimes L \rightarrow H^{\otimes 2} \otimes_H L$, called the pseudobracket, satisfying the following properties:

    - (skew-symmetry) $[y * x ]_L = - (\sigma_{12} \otimes_H 1) [x * y]_L$, for all $x, y \in L$,

    - (Jacobi identity) for any $x, y, z \in L$,
    \begin{align}
        [[x * y]_L * z]_L  =  [x * [y * z]_L ]_L - [ y * [x * z]_L]_L.
    \end{align}
    We denote a Lie $H$-pseudoalgebra as above by the pair $(L, [\cdot * \cdot]_L)$ or simply by $L$.
\end{definition}

\begin{remark}
    (i) When $H = {\bf k}$ is the base field, a Lie $H$-pseudoalgebra is simply a Lie algebra over ${\bf k}$.

    (ii) When $H = {\bf k}[\partial]$ is the algebra of polynomials in one variable $\partial$, a Lie $H$-pseudoalgebra is nothing but a Lie conformal algebra \cite{baka-kac-voro}.
\end{remark}

Let $(L, [\cdot * \cdot]_L)$ and $(L', [\cdot * \cdot]_{L'})$ be two Lie $H$-pseudoalgebras. A {\em homomorphism} of Lie $H$-pseudoalgebras from $L$ to $L'$ is a $H$-linear map $\Theta : L \rightarrow L'$ that preserve the pseudobrackets, i.e. $ [\Theta (x) * \Theta (y)]_{L'} = (\mathrm{id}_{H^{\otimes 2}} \otimes_H \Theta) [x * y]_L $, for all $x, y \in L.$ Further, it is said to be an isomorphism if $\Theta$ is an isomorphism of left $H$-modules.

Let $(L, [\cdot * \cdot]_L)$ be a Lie $H$-pseudoalgebra. A {\em representation} of $(L, [\cdot * \cdot]_L)$ is a left $H$-module $M$ equipped with a $H^{\otimes 2}$-linear map $ \cdot * \cdot : L \otimes M \rightarrow H^{\otimes 2} \otimes_H M$, called the action map, that satisfies
\begin{align*}
    [x * y]_L * u = x * (y * u) - y * (x * u), \text{ for all } x, y \in L, u \in M.
\end{align*}
It follows that the Lie $H$-pseudoalgebra $L$ is a representation of itself with the action map given by the pseudobracket $[\cdot * \cdot]_L$. This is called the adjoint representation.

Let $(L, [\cdot * \cdot]_L)$ be a Lie $H$-pseudoalgebra and $M$ be a representation of it with the action map being given by $ \cdot * \cdot: L \otimes M \rightarrow H^{\otimes 2} \otimes_H M$. Then there is a cochain complex $\{ C^\bullet (L, M), \delta \}$, where the cochain groups are given by
\begin{align*}
     C^n (L, M) = \begin{cases}
        \mathbf{k} \otimes_H M & \text{ if } n=0,\\
        \{ \theta \in  \mathrm{Hom}_{H^{\otimes n}} (L^{\otimes n}, H^{\otimes n} \otimes_H M)~|~ \theta \text{ is skew-symmetric} \} & \text{ if } n \geq 1.
        \end{cases}
\end{align*}
The coboundary map $\delta :  C^n (L, M) \rightarrow  C^{n+1} (L, M)$ is given by
\begin{align*}
   ( \delta (1 \otimes_H u)) (x) = \sum_{i} \varepsilon (g_i) f_i \otimes_H u_i, ~~ \text{ if } x * u = \sum_i (f_i \otimes g_i) \otimes_H u_i
\end{align*}
and
\begin{align*}
    (\delta \theta) (x_1, \ldots, x_{n+1} ) =~& \sum_{i=1}^{n+1} (-1)^{i} (\sigma_{1 \rightarrow i} \otimes_H 1) ~ x_i * \theta (x_1, \ldots, \widehat{x_i} , \ldots, x_{n+1})\\
    +& \sum_{1 \leq i < j \leq n+1} (-1)^{i+j+1} (\sigma_{\substack{1 \rightarrow i \\
    2 \rightarrow j}} \otimes_H 1) ~ \theta ([x_i * x_j]_L, x_1, \ldots, \widehat{x_i}, \ldots, \widehat{x_j}, \ldots, x_{n+1}),
\end{align*}
for $1 \otimes_H u \in C^0 (L, M)$, $\theta \in C^{n \geq 1} (L, M)$ and $x, x_1, \ldots, x_{n+1} \in L$. Here $\sigma_{1 \rightarrow i}$ and $\sigma_{\substack{1 \rightarrow i \\
    2 \rightarrow j}}$ are respectively the permutations on $H^{\otimes n+1}$ given by
    \begin{align*}
       \sigma_{1 \rightarrow i}( h_i \otimes h_1 \otimes \cdots \otimes h_{i-1} \otimes h_{i+1} \otimes \cdots \otimes h_{n+1}) = h_1 \otimes \cdots \otimes h_{n+1},\\
       \sigma_{\substack{1 \rightarrow i \\
    2 \rightarrow j}} ( h_i \otimes h_j \otimes h_1 \otimes \widehat{h_i} \otimes \cdots \otimes \widehat{h_j} \otimes \cdots \otimes h_{n+1} ) = h_1 \otimes \cdots \otimes h_{n+1}.
    \end{align*}
    The cohomology groups of the cochain complex $\{ C^\bullet (L, M), \delta \}$ are called the {\em cohomology of $L$ with coefficients in the representation $M$}. We denote the corresponding cohomology groups by $H^\bullet (L, M).$

    \begin{remark}
        (i) When $H = {\bf k}$, the above-defined cohomology reduces to the Chevalley-Eilenberg cohomology of the Lie algebra $L$ with coefficients in the Lie algebra representation $M$. On the other hand, when $H = {\bf k} [\partial]$, it reduces to the cohomology of a Lie conformal algebra \cite{baka-kac-voro}.

        (ii) It has been observed in \cite{bakalov-andrea-kac} that the second cohomology group $H^2 (L, M)$ classifies the set of all equivalence classes of abelian extensions of the Lie $H$-pseudoalgebra $L$ by the representation $M$.
    \end{remark}

    Let $L$ be a left $H$-module. Then it has been observed in \cite{wu} that the (shifted) graded space $C^{\bullet + 1} (L,L) = \oplus_{n=0}^\infty C^{n+1} (L,L)$, where
    \begin{align*}
        C^n(L,L) = \{ \theta \in \mathrm{Hom}_{H^{\otimes n}} (L^{\otimes n}, H^{\otimes n} \otimes_H L) |~ \theta \text{ is skew-symmetric}\}
    \end{align*}
    inherits a graded Lie bracket. This bracket generalizes the classical Nijenhuis-Richardson bracket and the bracket is explicitly given by
    \begin{align*}
        &  \qquad \llbracket P, Q \rrbracket (x_1, \ldots, x_{m+n+1} ) = (i_P Q - (-1)^{mn} i_Q P)  (x_1, \ldots, x_{m+n+1} ), \text{ where }\\
        &(i_P Q) (x_1, \ldots, x_{m+n+1} ) \\
        & \qquad = \sum_{\sigma \in \mathrm{Sh}(m+1, n)} (-1)^\sigma (\sigma_{i \rightarrow \sigma (i)} \otimes_H 1) ~Q \big(  P (   x_{\sigma (1)}, \ldots, x_{\sigma (m+1)}  ), x_{\sigma (m+2)} , \ldots, x_{\sigma (m+n+1)}   \big),   
    \end{align*}
    for $P \in C^{m+1} (L,L)$ and $Q \in C^{n+1}(L,L)$. Here $\sigma_{i \rightarrow \sigma (i)}$ is the permutation on $H^{\otimes m+n+1}$ given by
\begin{align*}
    \sigma_{i \rightarrow \sigma (i)} (h_{\sigma (1)} \otimes \cdots \otimes h_{\sigma (m+n+1)}) = h_1 \otimes \cdots \otimes h_{m+n+1}.
\end{align*}
Further, it has been shown that there is a one-to-one correspondence between Lie $H$-pseudoalgebra structures on the left $H$-module $L$ and Maurer-Cartan elements of the graded Lie algebra $( C^{\bullet + 1} (L, L), \llbracket ~, ~ \rrbracket   ).$ With the above notation, the coboundary map $\delta : C^n (L,L) \rightarrow C^{n+1} (L,L)$  of a Lie $H$-pseudoalgebra $(L, [\cdot * \cdot]_L)$ with coefficients in the adjoint representation is given by
\begin{align*}
    \delta \theta = \llbracket \rho_L, \theta \rrbracket, \text{ for } \theta \in C^n (L,L),
\end{align*}
where $\rho_L \in C^2 (L,L)$, $\rho_L (x, y) := [x*y]_L$ is the Maurer-Cartan element corresponding to the Lie $H$-pseudoalgebra structure on $L.$

\section{Extensions of Lie {\em H}-pseudoalgebras and the non-abelian cohomology}\label{sec3}

In this section, we study non-abelian extensions of a given Lie $H$-pseudoalgebra by another Lie $H$-pseudoalgebra. We introduce the non-abelian cohomology group (of a Lie $H$-pseudoalgebra with values in another Lie $H$-pseudoalgebra) that classifies equivalence classes of non-abelian extensions.

\begin{definition}\label{non-ab-ext-defn}
     Let $(L, [\cdot * \cdot]_L)$ and $(M, [\cdot * \cdot]_M)$ be two Lie $H$-pseudoalgebras.

     (i) A {\em non-abelian extension} of $(L, [\cdot * \cdot]_L)$ by $(M, [\cdot * \cdot]_M)$ is a new Lie $H$-pseudoalgebra $(E, [\cdot * \cdot]_E)$ equipped with a short exact sequence of Lie $H$-pseudoalgebras
     \begin{align}\label{nab-seq}
     \xymatrix{
     0 \ar[r] & (M, [\cdot * \cdot]_M) \ar[r]^i & (E, [\cdot * \cdot]_E) \ar[r]^p & (L, [\cdot * \cdot]_L) \ar[r] & 0.
     }
     \end{align}
     We denote a non-abelian extension as above simply by $ (E, [\cdot * \cdot]_E)$ if the short exact sequence is clear from the context.

     (ii) Let  $(E, [\cdot * \cdot]_E)$ and $ (E', [\cdot * \cdot]_{E'})$ be two non-abelian extensions of $(L, [\cdot * \cdot]_L)$ by $(M, [\cdot * \cdot]_M)$. These two non-abelian extensions are said to be {\em equivalent} if there exists a morphism $\Theta: E \rightarrow E'$ of Lie $H$-pseudoalgebras making the following diagram commutative
     \[
     \xymatrix{
     0 \ar[r] & (M, [\cdot * \cdot]_M) \ar@{=}[d] \ar[r]^i & (E, [\cdot * \cdot]_E) \ar[d]^\Theta \ar[r]^p & (L, [\cdot * \cdot]_L) \ar@{=}[d] \ar[r] & 0\\
     0 \ar[r] & (M, [\cdot * \cdot]_M) \ar[r]_{i'} & (E', [\cdot * \cdot]_{E'}) \ar[r]_{p'} & (L, [\cdot * \cdot]_L) \ar[r] & 0.
     }
     \]
\end{definition}

We denote by $\mathrm{Ext}_{nab} (L, M)$ the set of all equivalence classes of non-abelian extensions of $(L, [\cdot * \cdot]_L)$ by $(M, [\cdot * \cdot]_M)$.

\medskip

Let $(E, [\cdot * \cdot]_E)$ be a non-abelian extension of  $(L, [\cdot * \cdot]_L)$ by $(M, [\cdot * \cdot]_M)$ as of (\ref{nab-seq}). An $H$-linear map $s: L \rightarrow E$ is said to be a {\em section} of $p$ if the map $s$ satisfies $p \circ s = \mathrm{id}_L$. Note that a section of $p$ always exists. For any section $s : L \rightarrow E$, we define maps 
\begin{align*}
\chi \in \mathrm{Hom}_{H^{\otimes 2}} (L \otimes L, H^{\otimes 2} \otimes_H M) \quad \text{ and } \quad \psi \in \mathrm{Hom}_{H^{\otimes 2}} (L \otimes M, H^{\otimes 2} \otimes_H M) ~~\text{ by }
\end{align*}
\begin{align}\label{chi-psi}
    \chi (x, y) := [s(x) * s(y)]_E - ( \mathrm{id}_{H^{\otimes 2}} \otimes_H s) [x * y]_L ~~~ \text{ and } ~~~
    \psi (x, u) := [s(x) * u]_E,
\end{align}
for all $x, y \in L$ and $u \in M$. Note that the map $\psi$ satisfies
\begin{align}\label{deri-iden}
    [\psi (x, u) * v]_M = \psi (x, [u * v]_M) - [u * \psi (x , v)]_M, \text{ for } x \in L \text{ and } u, v \in M.
\end{align}
Further, we have the following.
\begin{proposition}
    The maps $\chi$ and $\psi$ satisfy
    \begin{align}\label{first-iden}
        \psi (x, \psi (y, u) ) - \psi (y, \psi (x, u)) - \psi ([x * y]_L, u) = [\chi (x, y) * u]_M,
        \end{align}
        \begin{align}\label{second-iden}
        \psi (x, \chi (y, z)) - \psi (y, \chi (x,z)) + \psi (z, \chi (x, y)) + \chi (x, [y * z]_L) - \chi (y, [x * z]_L) + \chi (z, [x * y]_L) = 0,
    \end{align}
    for all $x, y, z \in L$ and $u \in M$.
\end{proposition}

\begin{proof}
    For all $x, y \in L$ and $u \in M$, we have
    \begin{align*}
        &\psi (x, \psi (y, u)) - \psi (y, \psi (x, u)) - \psi ([x * y]_L, u) \\
        &= [s(x) * [s(y) * u]_E]_E - [s(y) * [s(x) * u]_E]_E - [   (\mathrm{id}_{H^{\otimes 2}} \otimes_H s) [x * y]_L * u] \\
        &= [[ s(x) * s(y)]_E * u] -  [   (\mathrm{id}_{H^{\otimes 2}} \otimes_H s) [x * y]_L * u] = [\chi (x, y) * u]_E.
    \end{align*}
    This proves the identity (\ref{first-iden}). Moreover, for any $x, y, z \in L$, we have
    \begin{align}\label{first}
        & \psi (x, \chi (y, z)) - \psi (y, \chi (x,z)) + \psi (z, \chi (x, y)) \nonumber\\
        & = [s(x) * [s(y) * s(z)]_E]_E - [s(x) * (\mathrm{id}_{H^{\otimes 2}} \otimes_H s) [y*z]_L]_E -  [s(y) * [s(x) * s(z)]_E]_E \\
        & ~~ + [s(y) * (\mathrm{id}_{H^{\otimes 2}} \otimes_H s) [x*z]_L]_E - [[s(x) * s(y)]_E * s(z)]_E - [s(z) * (\mathrm{id}_{H^{\otimes 2}} \otimes_H s) [x*y]_L]_E \nonumber
    \end{align}
    and
    \begin{align}\label{second}
         &\chi (x, [y * z]_L) - \chi (y, [x * z]_L) + \chi (z, [x * y]_L) \nonumber \\
         &= [s(x) * (\mathrm{id}_{H^{\otimes 2}} \otimes_H s)[y * z]_L]_E - (\mathrm{id}_{H^{\otimes 3}} \otimes_H s) [x*[y*z]_L]_L - [s(y) * (\mathrm{id}_{H^{\otimes 2}} \otimes_H s)[x * z]_L]_E \\
         & ~~ + (\mathrm{id}_{H^{\otimes 3}} \otimes_H s) [y*[x*z]_L]_L + [s(z) * (\mathrm{id}_{H^{\otimes 2}} \otimes_H s)[x*y]_L]_E + (\mathrm{id}_{H^{\otimes 3}} \otimes_H s) [[x*y]_L*z]_L. \nonumber
    \end{align}
    By summing up the expressions of (\ref{first}) and (\ref{second}), we obtain the desired identity (\ref{second-iden}).
\end{proof}

Note that the maps $\chi$ and $\psi$ both depend on the section $s$. Let $s'$ be any other section of $p$. We define a map $\varphi : L \rightarrow E$ by $\varphi (x) := s(x) - s'(x)$, for $x \in L$. Then we have $\mathrm{im} (\varphi) \subset \mathrm{ker} (p) = \mathrm{im} (i)$. Therefore, $\varphi$ can be realized as a map $\varphi : L \rightarrow M$. Let $\chi'$ and $\psi'$ be the maps induced by the section $s'$ as of (\ref{chi-psi}). Then we have
\begin{align*}
    \psi (x, u) - \psi' (x, u) = [(s(x) - s'(x)) * u]_E = [\varphi (x) * u]_M,
    \end{align*}
    \begin{align*}
    \chi (x, y) - \chi'(x, y) =~& [s(x) * s(y)]_E - (\mathrm{id}_{H^{\otimes 2}} \otimes_H s) [x*y]_L -  [s'(x) * s'(y)]_E + (\mathrm{id}_{H^{\otimes 2}} \otimes_H s') [x*y]_L \\
    =~& [ (\varphi + s')(x) * (\varphi + s')(y)]_E - (\mathrm{id}_{H^{\otimes 2}} \otimes_H (\varphi + s')) [x*y]_L \\
    &-  [s'(x) * s'(y)]_E + (\mathrm{id}_{H^{\otimes 2}} \otimes_H s') [x*y]_L \\
    =~& \psi' (x, \varphi(y)) - (\sigma_{12} \otimes_H 1) \psi' (y, \varphi (x)) - (\mathrm{id}_{H^{\otimes 2}} \otimes \varphi) [x * y]_L + [\varphi (x) * \varphi (y)]_M.
\end{align*}
The above discussions motivate us to consider the following definition.

\begin{definition}
    Let $(L, [\cdot * \cdot]_L)$ and $(M, [\cdot * \cdot]_M)$ be two Lie $H$-pseudoalgebras.

    (i) A {\em non-abelian $2$-cocycle} of $(L, [\cdot * \cdot]_L)$ with values in $(M, [\cdot * \cdot]_M)$ is a pair $(\chi, \psi)$ consisting of maps $\chi \in \mathrm{Hom}_{H^{\otimes 2}} (L \otimes L, H^{\otimes 2} \otimes_H M)$ and $\psi \in \mathrm{Hom}_{H^{\otimes 2}} (L \otimes M, H^{\otimes 2} \otimes_H M)$ in which $\chi$ is skew-symmetric and satisfying the conditions (\ref{deri-iden}), (\ref{first-iden}) and (\ref{second-iden}).

    (ii) Let $(\chi, \psi)$ and $(\chi', \psi')$ be two non-abelian $2$-cocycles of $(L, [\cdot * \cdot]_L)$ with values in $(M, [\cdot * \cdot]_M)$. These two cocycles are said to be {\em equivalent} if there exists a map $\varphi \in \mathrm{Hom}_H (L, H \otimes_H M) = \mathrm{Hom}_H (L, M)$ such that for all $x, y \in L$ and $u \in M$,
    \begin{align}
        \psi (x, u) - \psi' (x, u) =~& [\varphi (x) * u]_M,\label{equiv1}\\
        \chi (x, y) - \chi' (x, y) =~& \psi' (x, \varphi(y)) - (\sigma_{12} \otimes_H 1) \psi' (y, \varphi (x)) - (\mathrm{id}_{H^{\otimes 2}} \otimes \varphi) [x * y]_L + [\varphi (x) * \varphi (y)]_M. \label{equiv2}
    \end{align}
\end{definition}

\medskip

We denote by $H^2_{nab} (L, M)$ the set of all equivalence classes of non-abelian $2$-cocycles of $(L, [\cdot * \cdot]_L)$ with values in $(M, [\cdot * \cdot]_M)$. We call $H^2_{nab} (L, M)$ as the {\em non-abelian cohomology group} of $(L, [\cdot * \cdot]_L)$ with values in $(M, [\cdot * \cdot]_M)$. In the following result, we classify the set $\mathrm{Ext}_{nab} (L, M)$ of all equivalence classes of non-abelian extensions by the non-abelian cohomology group $H^2_{nab} (L, M)$. More precisely, we have the following.

\begin{theorem}\label{nab-theorem}
    Let $(L, [\cdot * \cdot]_L)$ and $(M, [\cdot * \cdot]_M)$ be two Lie $H$-pseudoalgebras. Then there is a bijection
    \begin{align*}
        \mathrm{Ext}_{nab} (L, M) \cong H^2_{nab} (L, M).
    \end{align*}
\end{theorem}

\begin{proof}
    Let $(E, [\cdot * \cdot]_E)$ be a non-abelian extension of $L$ by $M$ as of (\ref{nab-seq}). For any section $s$, let $(\chi, \psi)$ be the corresponding non-abelian $2$-cocycle of $L$ with values in $M$. Suppose $(E', [\cdot * \cdot]_{E'})$ is another non-abelian extension of $L$ by $M$ equivalent to $(E, [\cdot * \cdot]_E)$. See Definition \ref{non-ab-ext-defn} (ii). Then we have $p'\circ (\Theta \circ s) = (p'\circ \Theta) \circ s = p \circ s = \mathrm{id}_L$ which shows that $s'= \Theta \circ s : L \rightarrow E'$ is a section of the map $p'$. If $(\chi', \psi')$ is the non-abelian $2$-cocycle corresponding to the extension $(E', [\cdot * \cdot]_{E'})$ with the section $s'$, then we have
    \begin{align*}
        \chi' (x, y) =~& [s'(x) * s'(y)]_{E'} - (\mathrm{id}_{H^{\otimes 2}} \otimes_H s') [x*y]_L \\
        =~& [ (\Theta \circ s)(x) * (\Theta \circ s)(y)]_{E'} - (\mathrm{id}_{H^{\otimes 2}} \otimes_H \Theta \circ s) [x * y]_L
        \\
       =~&  (\mathrm{id}_{H^{\otimes 2}} \otimes_H \Theta ) [s(x) * s(y)]_E - (\mathrm{id}_{H^{\otimes 2}} \otimes_H \Theta) (\mathrm{id}_{H^{\otimes 2}} \otimes_H s) [x*y]_L  \\ 
         =~& (\mathrm{id}_{H^{\otimes 2}} \otimes_H \Theta ) \chi (x, y) = \chi (x, y) \quad (\text{as } \Theta|_M = \mathrm{id}_M ), \text{ and } \\
        \psi' (x, u) =~& [s'(x) * u]_{E'} = [(\Theta \circ s)(x) * \Theta (u)]_{E'} ~~~ (\text{as } \Theta (u) = u) \\
        =~& (\mathrm{id}_{H^{\otimes 2}} \otimes_H \Theta) [s(x) * u]_{E} \\
        =~&  (\mathrm{id}_{H^{\otimes 2}} \otimes_H \Theta) \psi (x, u) = \psi (x, u)  \quad (\text{as } \Theta|_M = \mathrm{id}_M ).
    \end{align*}
    for $x \in L$, $u \in M$. Hence the non-abelian $2$-cocycles $(\chi, \psi)$ and $(\chi', \psi')$ are the same. This shows that there is a well-defined map $\Upsilon: \mathrm{Ext}_{nab} (L, M) \rightarrow H^2_{nab} (L, M)$ which assigns an equivalence class of extensions to the equivalence class of the corresponding non-abelian $2$-cocycles.

    \medskip

    Conversely, let $(\chi, \psi)$ be a non-abelian $2$-cocycle. Consider the left $H$-module $L \oplus M$ together with the pseudobracket
    \begin{align}\label{direct-pseudo}
        [(x, u) * (y, v)]_{(\chi, \psi)} := \big(  [x* y]_L, \psi (x, v) - (\sigma_{12} \otimes_H 1) \psi (y, u) + \chi (x, y) + [u * v]_M   \big),
    \end{align}
    for $(x, u), (y, v) \in L \oplus M$. The pseudobracket is obviously $H^{\otimes 2}$-linear and skew-symmetric. Further, it follows from the properties of $\chi$ and $\psi$ that the pseudobracket (\ref{direct-pseudo}) also satisfies the Jacobi identity. Hence $(L \oplus M, [\cdot * \cdot]_{(\chi,\psi)})$ is a Lie $H$-pseudoalgebra, denoted simply by $L \oplus_{(\chi, \psi)} M$. Moreover, the short exact sequence $0 \rightarrow M \xrightarrow{i} L \oplus_{(\chi, \psi)} M \xrightarrow{p} L \rightarrow 0$ becomes a non-abelian extension of $L$ by $M$, where $i (u) = (0, u)$ $p (x, u)= x$, for $u \in M$ and $(x, u) \in L \oplus M$.

    Next, let $(\chi, \psi)$ and $(\chi', \psi')$ be two equivalent non-abelian $2$-cocycles. In other words, there exists a map $\varphi \in \mathrm{Hom}_H (L, H \otimes_H M) = \mathrm{Hom}_H (L, M)$ such that (\ref{equiv1}) and (\ref{equiv2}) holds. Let $L \oplus_{(\chi, \psi)} M$ and $L \oplus_{(\chi', \psi')} M$  be the Lie $H$-pseudoalgebras induced by the $2$-cocycles $(\chi, \psi)$ and $(\chi', \psi')$, respectively. We define a map 
    \begin{align*}
    \Theta : L \oplus_{(\chi, \psi)} M \rightarrow L \oplus_{(\chi', \psi')} M ~~ \text{ by } ~~ \Theta ((x, u)) = (x, \varphi (x) + u), \text{ for } (x, u) \in L \oplus M.
    \end{align*}
    Then we have
    \begin{align*}
      &( \mathrm{id}_{H^{\otimes 2}} \otimes_H  \Theta) [(x, u) * (y, v)]_{(\chi, \psi)} \\
     & = ( \mathrm{id}_{H^{\otimes 2}} \otimes_H  \Theta)  ([x * y]_L , \psi (x, v) - (\sigma_{12} \otimes_H 1) \psi (y, u) + \chi (x, y) + [u * v]_M)
      \\ 
     & =  ( \mathrm{id}_{H^{\otimes 2}} \otimes_H  \Theta)  \big( [x * y]_L , \psi' (x, v) + [\varphi (x) * v]_M -  (\sigma_{12} \otimes_H 1)  \psi' (y, u) - (\sigma_{12} \otimes_H 1)  [\varphi (y) * u]_M \\
     & \quad + \chi' (x, y) + \psi' (x, \varphi(y)) - \psi' (y, \varphi (x)) - (\mathrm{id}_{H^{\otimes 2}} \otimes \varphi) [x * y]_L + [\varphi (x) * \varphi (y)]_M + [u * v]_M  \big)\\
     & = \big(  [x* y]_L, \psi' (x , \varphi (y) + v) -   (\sigma_{12} \otimes_H 1) \psi' (y, \varphi (x) + v) + \chi' (x, y) + [   (\varphi (x) + u) * (\varphi (y) + v)]_M \big) \\
     & = [   (x, \varphi (x) + u) * (y, \varphi (y) + v)]_{(\chi', \psi')} \\
      &  = [\Theta (x, u) * \Theta (y, v)]_{(\chi', \psi')}.
    \end{align*}
    Hence the map $\Theta$ defines an equivalence of non-abelian extensions between $L \oplus_{(\chi, \psi)} M $ and $L \oplus_{(\chi', \psi')} M $. Therefore, we obtain a well-defined map $\Omega : H^2_{nab} (L, M) \rightarrow \mathrm{Ext}_{nab} (L, M)$.

    Finally, it is straightforward to verify that the maps $\Upsilon$ and $\Omega$ are inverses to each other. Hence we obtain a bijection between $\mathrm{Ext}_{nab} (L, M)$ and $ H^2_{nab} (L, M)$.
\end{proof}

Let $(L, [\cdot * \cdot]_L)$ be a Lie $H$-pseudoalgebra and $M$ be a left $H$-module (not necessarily a Lie $H$-pseudoalgebra nor any representation of $L$). We realize $M$ as a Lie $H$-pseudoalgebra with the trivial pseudobracket $[\cdot * \cdot]_M = 0$. An {\em abelian extension} of the Lie $H$-pseudoalgebra $(L, [\cdot * \cdot]_L)$  by the left $H$-module $M$ is a short exact sequence of Lie $H$-pseudoalgebras
\begin{align*}
    \xymatrix{
    0 \ar[r] & (M , [\cdot * \cdot]_M  = 0) \ar[r]
^i   &   (E , [\cdot * \cdot]_E) \ar[r]^p & (L, [\cdot * \cdot]_L) \ar[r] & 0.
}
\end{align*}
One can also define equivalences between abelian extensions of $(L, [\cdot * \cdot]_L)$  by the left $H$-module $M$.

Let $s: L \rightarrow E$ be any section of the map $p$. Then there is a representation of the Lie $H$-pseudoalgebra $(L, [\cdot * \cdot]_L)$ on the left $H$-module $M$ with the action map given by $x * u := \psi (x, u) = [s(x) * u]_E$, for $x \in L$ and $u \in M$. This is indeed a representation that follows from (\ref{first-iden}). Further, this representation does not depend on the choice of the section $s$. 

Next, let $M$ be a given representation of the Lie $H$-pseudoalgebra $(L, [\cdot * \cdot]_L)$. We denote by $\mathrm{Ext}_{ab} (L, M)$ the set of all equivalence classes of abelian extensions of $(L, [\cdot * \cdot]_L)$ by the left $H$-module $M$ so that the induced representation on $M$ coincides with the given one. With the above definition and from our Theorem \ref{nab-theorem}, we recover the classification result about abelian extensions \cite{bakalov-andrea-kac}.

\begin{theorem}
    Let $(L, [\cdot * \cdot]_L)$ be a Lie $H$-pseudoalgebra and $M$ be a representation of it. Then
    \begin{align*}
        \mathrm{Ext}_{ab} (L, M) \cong H^2 (L, M).
    \end{align*}
\end{theorem}

\begin{proof}
    Let $(E, [\cdot * \cdot]_E)$ be an abelian extension of $(L, [\cdot * \cdot]_L)$ by $M$ representing an element of $\mathrm{Ext}_{ab} (L, M)$. Since $[\cdot * \cdot]_M = 0$, the condition (\ref{deri-iden}) is redundant. On the other hand, it follows from (\ref{first-iden}) that $\psi$ defines a representation of $(L, [\cdot * \cdot]_L)$ on the left $H$-module $M$ which is the prescribed one. Finally, the condition (\ref{second-iden}) yields that $\chi$ is a $2$-cocycle on the cochain complex of the Lie $H$-pseudoalgebra $(L, [\cdot * \cdot]_L)$ with coefficients in the representation $M$. Moreover, if $(E, [\cdot * \cdot]_E)$ and $(E', [\cdot * \cdot]_{E'})$ are two equivalent abelian extensions, then the corresponding $2$-cocycles are cohomologous. Hence, in this case, we obtain a map $\Upsilon : \mathrm{Ext}_{ab} (L, M) \rightarrow H^2(L, M).$

    Conversely, if $\chi$ is a $2$-cocycle (i.e. $\chi \in C^2 (L, M)$ with $\delta (\chi) = 0$), then we can define  a Lie $H$-pseudoalgebra structure on the left $H$-module $L \oplus M$ (denote the structure by $L \oplus_\chi M$) similar to (\ref{direct-pseudo}) with the fact that $[ \cdot * \cdot ]_M = 0$. This Lie $H$-pseudoalgebra makes $0 \rightarrow M \rightarrow L \oplus_\chi M \rightarrow L \rightarrow 0$ into an abelian extension of the Lie $H$-pseudoalgebra $L$ by $M$. Finally, if $\chi'$ is another $2$-cocycle cohomologous to $\chi$ (i.e. $\delta (\chi') = 0$ and $\chi - \chi' = \delta (\varphi)$ for some $\varphi \in C^1 (L, M)$) then the abelian extensions $L \oplus_\chi M$ and $L \oplus_{\chi'} M$ are equivalent. Hence one gets a map $\Omega : H^2 (L, M) \rightarrow \mathrm{Ext}_{ab} (L, M)$. Finally, the maps $\Upsilon$ and $\Omega$ are inverses to each other. This completes the proof.
\end{proof}

\section{Non-abelian cohomology as Deligne groupoid}\label{sec4}
In this section, we show that the non-abelian cohomology group $H^2_{nab} (L, M)$ of a Lie $H$-pseudoalgebra $L$ with values in another Lie $H$-pseudoalgebra $M$ can be seen as the Deligne groupoid of a suitable differential graded Lie algebra $\mathfrak{g}$. More precisely, we show that the set of all non-abelian cocycles has a bijection with the set of all Maurer-Cartan elements of $\mathfrak{g}$. Further, the notion of equivalence between non-abelian cocycles coincides with the equivalence between the corresponding Maurer-Cartan elements. In the end, we also discuss abelian extensions in terms of the tangent cohomology complex.


First recall that a {\em differential graded Lie algebra} (in short {dgLa}) is a triple $(\mathfrak{g}, \llbracket~,~\rrbracket, d)$ in which $\mathfrak{g} = \oplus_{i \in \mathbb{Z}} \mathfrak{g}^i$ is a graded vector space equipped with a degree $0$ bilinear bracket $ \llbracket ~,~ \rrbracket : \mathfrak{g} \times \mathfrak{g} \rightarrow \mathfrak{g}$ satisfying

- (graded skew-symmetry) $ \llbracket x,y \rrbracket = - (-1)^{|x| |y|} \llbracket y, x \rrbracket,$

- (graded Jacobi identity) $\llbracket x, \llbracket y, z \rrbracket \rrbracket = \llbracket \llbracket x, y \rrbracket, z \rrbracket + (-1)^{|x| |y|} \llbracket y, \llbracket x,z \rrbracket \rrbracket,$

\noindent and a degree $1$ linear map $d : \mathfrak{g} \rightarrow \mathfrak{g}$ satisfying $d^2 = 0$ and the following Leibniz rule
\begin{align*}
    d \llbracket x, y \rrbracket = \llbracket dx, y \rrbracket  + (-1)^{|x|} \llbracket x, dy \rrbracket, \text{for all homogeneous } x, y, z \in \mathfrak{g}.
\end{align*}

Let $(\mathfrak{g}, \llbracket ~,~ \rrbracket, d)$ be a differential graded Lie algebra. An element $\alpha \in \mathfrak{g}^1$ is said to be a {\em Maurer-Cartan element} of $(\mathfrak{g}, \llbracket ~,~\rrbracket, d)$ if the element $\alpha$ satisfies
\begin{align*}
    d \alpha + \frac{1}{2} \llbracket \alpha, \alpha \rrbracket = 0.
\end{align*}
We denote the set of all Maurer-Cartan elements of $(\mathfrak{g}, \llbracket ~,~\rrbracket , d)$ simply by the notation $MC (\mathfrak{g})$.

There is an equivalence relation on the set $MC (\mathfrak{g})$. First, recall that an element $\beta \in \mathfrak{g}^0$ is said to be {\em ad-nilpotent} if for any $x \in \mathfrak{g}$, there exists a natural number $n$ such that $(ad_\beta)^n (x) = 0$, where $ad_\beta = \llbracket \beta , - \rrbracket$. Note that, if $\beta$ is ad-nilpotent then one can make sense the operator $e^{ad_\beta}$. For an ad-nilpotent element $\beta$, we consider the element $g_\beta = - \sum_{n \in \mathbb{N} \cup \{ 0 \} } \frac{1}{(n+1)!} (ad_\beta)^n d \beta.$

Let $(\mathfrak{g}, \llbracket ~,~ \rrbracket, d)$ be a differential graded Lie algebra with abelian $\mathfrak{g}^0$. Let $\alpha, \alpha' \in MC (\mathfrak{g})$ be two Maurer-Cartan elements. They are said to be {\em equivalent} (and we write $\alpha \sim \alpha'$) if there exists an ad-nilpotent element $\beta$ such that
\begin{align*}
    \alpha' = e^{ad_\beta} \alpha + g_\beta.
\end{align*}
It has been observed in \cite[Lemma 1.3]{fregier} that $\sim$ defines an equivalence relation on $MC (\mathfrak{g})$.

\begin{definition}
    Let $(\mathfrak{g}, \llbracket ~,~ \rrbracket, d)$ be a differential graded Lie algebra with abelian $\mathfrak{g}^0$. The {\em Deligne groupoid} of $\mathfrak{g}$ is the groupoid whose set of objects is $MC (\mathfrak{g})$ and whose set of Homs is empty for non-equivalent objects and reduced to one element for equivalent objects.
\end{definition}

We denote the set of all connected components of the Deligne groupoid by $\mathcal{MC}(\mathfrak{g})$. Thus,
\begin{align*}
    \mathcal{MC} (\mathfrak{g}) := MC (\mathfrak{g}) / \sim.
\end{align*}

\medskip

Let $(L, [\cdot * \cdot]_L )$ and $(M, [\cdot * \cdot]_M )$ be two Lie $H$-pseudoalgebras. Consider the direct sum left $H$-module $L \oplus M$ and the associated graded Lie algebra
\begin{align*}
    C^{\bullet + 1} (L \oplus M, L \oplus M) := \big( \oplus_{n=0}^\infty C^{n+1} (L \oplus M, L \oplus M), \llbracket ~, ~ \rrbracket   \big)
\end{align*}
described in Section \ref{sec2}. Since $L$ and $M$ are both Lie $H$-pseudoalgebras, the direct sum $L \oplus M$ inherits a Lie $H$-pseudoalgebra structure. This Lie $H$-pseudoalgebra structure gives rise to a Maurer-Cartan element (namely $\rho_L + \rho_M$) on the above graded Lie algebra. Hence we obtain a differential 
\begin{align*}
d : C^n (L \oplus M, L \oplus M) \rightarrow  C^{n+1} (L \oplus M, L \oplus M), ~~d := \llbracket \rho_L + \rho_M , - \rrbracket
\end{align*}
which makes the triple $(C^{\bullet + 1} (L \oplus M, L \oplus M), \llbracket ~, ~ \rrbracket, d)$ into a differential graded Lie algebra.

On the other hand, observe that the left $H$-module $M$ can be equipped with a representation of the Lie $H$-pseudoalgebra $L \oplus M$ with the action map
\begin{align*}
    (x, u) * v := [u * v]_M, \text{ for } (x,u) \in L \oplus M \text{ and } v \in M.
\end{align*}
Hence we may consider the cochain complex $\{ C^\bullet (L \oplus M, M), \delta \}$ defining the cohomology of the Lie $H$-pseudoalgebra $L \oplus M$ with coefficients in the representation $M$. We define a subcomplex $C^\bullet_{>} (L \oplus M, M)$ of the complex $C^\bullet (L \oplus M, M)$ by
\begin{align*}
    C^\bullet_{>} (L \oplus M, M) \cong \oplus_{(m,n) \in \mathbb{N} \times (\mathbb{N} \cup \{ 0 \})} C^{m,n},
\end{align*}
where $C^{m,n}$ is given by restricting elements of $C^{m+n} (L \oplus M, M)$ to $L^{\otimes m} \otimes M^{\otimes n}$.

For any $f \in C^{m,n}$, it is easy to see from the definition of the bracket $\llbracket ~, ~ \rrbracket$ that $\llbracket  \rho_L , f  \rrbracket \in C^{m+1,n}$ and $\llbracket \rho_M  , f  \rrbracket \in C^{m,n+1}$. Hence $d(f) = \llbracket  \rho_L  + \rho_M , f  \rrbracket = \llbracket \rho_L  ,   f \rrbracket + \llbracket \rho_M  , f  \rrbracket$ which shows that the map $d$ satisfies $d ( C^\bullet_{>} (L \oplus M, M)) \subset C^{\bullet + 1}_{>} (L \oplus M, M)$. On the other hand, for $f \in C^{m,n}$ and $g \in C^{p, q}$, we have $i_f g, i_g f \in C^{m+p, n+q-1}$. Hence
\begin{align*}
    \llbracket f, g \rrbracket = i_f g - (-)^{(m+n-1) (p+q -1)} i_g f ~\in C^{m+p,n+q-1}
\end{align*}
which implies that $C^\bullet_{>} (L\oplus M, M)$ is closed for the bracket $\llbracket ~, ~ \rrbracket$. This shows that $C^{\bullet +1 }_{>} (L \oplus M, M)$ is a differential graded Lie subalgebra of $(C^{\bullet + 1} (L \oplus M, L \oplus M), \llbracket ~, ~ \rrbracket, d)$. In the rest of this section, we will denote the differential graded Lie algebra $\big(  C^{\bullet + 1}_{>} (L \oplus M, M), \llbracket ~, ~ \rrbracket, d   \big)$ simply by $\mathfrak{g}$. Note that
\begin{align*}
    \mathfrak{g}^0 = C^1_{>} (L \oplus M, M) = C^{1,0} = \mathrm{Hom}_H (L,M)
\end{align*}
and it can be easily checked that $\mathfrak{g}^0$ is abelian and all elements of $\mathfrak{g}^0$ are ad-nilpotent. Moreover,
\begin{align*}
    \mathfrak{g}^1 = C^2_{>} (L \oplus M, M) = C^{2,0} \oplus C^{1,1}  ~~~ \text{ and } ~~~ \mathfrak{g}^2 = C^3_{>} (L \oplus M, M) = C^{3,0} \oplus C^{2,1} \oplus C^{1,2}.
\end{align*}



\begin{proposition}\label{deligne-2}
Let $L$ and $M$ be two Lie $H$-pseudoalgebras. 

(i) An element $\chi + \psi \in \mathfrak{g}^1 = C^{2,0} \oplus C^{1,1}$ is a Maurer-Cartan element of the differential graded Lie algebra $\mathfrak{g}$ if and only if the pair $(\chi, \psi)$ is a non-abelian $2$-cocycle of the Lie $H$-pseudoalgebra with values in $M$.
    
 (ii)  Two Maurer-Cartan elements $ \chi + \psi$ and $\chi'+\psi'$ are equivalent in the differential graded Lie algebra $\mathfrak{g}$ if and only if the corresponding non-abelian $2$-cocycles $(\chi, \psi)$ and $(\chi', \psi')$ are equivalent.
\end{proposition}

\begin{proof}
(i) For any $x, y, z \in L$ and $u, v \in M$, we observe that
\begin{align*}
    &\big(  d(\chi + \psi) + \frac{1}{2} \llbracket \chi + \psi , \chi + \psi \rrbracket \big) (x, u, v) \\
    &= \big( \llbracket \rho_L + \rho_M, \chi + \psi \rrbracket +  \frac{1}{2} \llbracket \chi + \psi , \chi + \psi \rrbracket \big) (x, u, v) \\
    &= \big( \llbracket \rho_L , \chi \rrbracket + \llbracket \rho_L, \psi \rrbracket + \llbracket \rho_M, \chi \rrbracket + \llbracket \rho_M, \psi \rrbracket + i_{\chi + \psi} (\chi + \psi)  \big) (x, u, v) \\
    &= \big( i_{\rho_L} \chi + i_\chi \rho_L + i_{\rho_L} \psi + i_\psi \rho_L + i_{\rho_M} \chi + i_\chi \rho_M + i_{\rho_M} \psi + i_\psi \rho_M + i_\chi \chi + i_\chi \psi + i_\psi \chi + i_\psi \psi \big) (x, u, v) \\
    &= - \psi (x, [u*v]_M) + [\psi (x, u) * v]_M + [u * \psi (x, v)]_M. 
\end{align*}
    Similarly, we have
    \begin{align*}
     &\big(  d(\chi + \psi) + \frac{1}{2} \llbracket \chi + \psi , \chi + \psi \rrbracket \big) (x, y, u) \\
     &\quad = \psi ( [x*y]_L, u) + [\chi (x, y) * u]_M + \psi (y, \psi (x, u)) - \psi (x, \psi (y, u))
     \end{align*}
     and
     \begin{align*}
      &\big(  d(\chi + \psi) + \frac{1}{2} \llbracket \chi + \psi , \chi + \psi \rrbracket \big) (x, y, z) \\
      & \quad = - \chi (x, [y * z]_L) + \chi (y, [x * z]_L) - \chi (z, [x * y]_L)    - (x, \chi (y, z)) + \psi (y, \chi (x,z)) - \psi (z, \chi (x, y)).
\end{align*}
It follows that $ \chi + \psi $ is a Maurer-Cartan element of the differential graded Lie algebra $\mathfrak{g}$ if and only if the pair $(\chi, \psi)$ is a non-abelian $2$-cocycle of the Lie $H$-pseudoalgebra $L$ with values in $M$.

\medskip

  (ii)  Note that two elements $c$ and $c'$ in $MC (\mathfrak{g})$ are equivalent if there exists $\beta \in \mathrm{Hom}_H (L, M)$ such that
    \begin{align*}
        c' = e^{ad_\beta} c + g_\beta.
    \end{align*}
    Let $c= \chi + \psi$ and take any two elements $e = x+ u$, $e' = y+v \in L \oplus M$. Then we have
    \begin{align*}
        (e^{ad_\beta} c ) (e,e') =~& (e^{ad_\beta} (\chi+\psi)) (e,e') \\
        =~& \big(  \chi+ \psi + \llbracket \beta , \chi+\psi \rrbracket + \underbrace{\frac{1}{2} \llbracket \beta, \llbracket \beta, \chi + \psi \rrbracket \rrbracket + \cdots }_{= 0} \big) (e, e') \\
        =~& \chi (e,e') + \psi (e,e') + \psi (x, \beta (y)) - (\sigma_{12} \otimes_H 1) \psi (y, \beta (x)) \\
        =~& \chi (x, y) + \psi (x, v) - (\sigma_{12} \otimes_H 1) \psi (y, u) + \psi (x, \beta (y)) - (\sigma_{12} \otimes_H 1) \psi (y, \beta (x)).
    \end{align*}
    On the other hand, we observe that
    \begin{align*}
        - (d \beta) (e, e') =~& - \llbracket \rho_L + \rho_M, \beta \rrbracket (e,e') \\
        =~& - (\mathrm{id}_{H^{\otimes 2}} \otimes_H \beta) [x * y]_L + [\beta (x) * v]_M + [u * \beta (y)]_M
    \end{align*}
    and $- \llbracket \beta, d \beta \rrbracket (e, e') = 2 [\beta (x) * \beta (y)]_M.$ Hence
    \begin{align*}
        (g_\beta) (e, e') =~& - \bigg( \sum_{n \in \mathbb{N} \cup \{ 0 \} }  \frac{1}{(n+1)!} (ad_\beta)^n d \beta \bigg) (e, e') \\
        =~& - (d \beta ) (e, e') - \frac{1}{2} ( ad_\beta d\beta) (e, e') \\
        =~& - (d \beta ) (e, e') - \frac{1}{2} \llbracket \beta, d \beta \rrbracket (e, e') \\
        =~& - (\mathrm{id}_{H^{\otimes 2}} \otimes_H \beta) [x * y]_L + [\beta (x) * v]_M + [u * \beta (y)]_M + [\beta (x) * \beta (y)]_M.
    \end{align*}
    Therefore, if $c' = \chi' + \psi'$, then $c$ and $c'$ are equivalent if and only if
    \begin{align}\label{c-dash}
        c' (e, e') =~& \chi (x, y) + \psi (x, v) - (\sigma_{12} \otimes_H 1) \psi (y, u) + \psi (x, \beta (y)) - (\sigma_{12} \otimes_H 1) \psi (y, \beta (x)) \\ ~& - (\mathrm{id}_{H^{\otimes 2}} \otimes_H \beta) [x * y]_L + [\beta (x) * v]_M + [u * \beta (y)]_M + [\beta (x) * \beta (y)]_M. \nonumber
    \end{align}
    However, since $c' = \chi' + \psi'$, we have $c' (e,e') = \chi' (x, y) + \psi' (x, v) - (\sigma_{12} \otimes_H 1) \psi' (y, u)$. Comparing this with the expression of (\ref{c-dash}), we get that the non-abelian $2$-cocycles $(\chi', \psi')$ and $(\chi, \psi)$ are equivalent, and the equivalence is given by the map $\beta$. Conversely, equivalences between non-abelian $2$-cocycles imply equivalences between the corresponding Maurer-Cartan elements.
\end{proof}


An immediate consequence of Proposition \ref{deligne-2} is given by the following.

\begin{theorem}
    Let $L$ and $M$ be two Lie $H$-pseudoalgebras. Then 
    \begin{align*}
        H^2_{nab} (L,M) \cong \mathcal{MC} (\mathfrak{g}).
    \end{align*}
\end{theorem}

Let $\mathfrak{g} = (\mathfrak{g}, \llbracket ~,~ \rrbracket, d)$ be a differential graded Lie algebra and $\alpha \in MC (\mathfrak{g})$ be a Maurer-Cartan element. Then one can construct a new differential graded Lie algebra $(\mathfrak{g}, \llbracket ~,~ \rrbracket_\alpha, d_\alpha)$ on the same graded vector space with the operations given by
\begin{align*}
    \llbracket x, y \rrbracket_\alpha := \llbracket x, y \rrbracket ~~~ \text{ and } ~~~ d_\alpha (x) = d(x) + \llbracket \alpha, x \rrbracket, \text{ for } x, y \in \mathfrak{g}.
\end{align*}
The differential graded Lie algebra $(\mathfrak{g}, \llbracket ~,~ \rrbracket_\alpha, d_\alpha)$ is denoted by $\mathfrak{g}_\alpha$ and it is called the tangent complex of $\mathfrak{g}$ at $\alpha$.

Let $(L, [\cdot * \cdot]_L)$ be a Lie $H$-pseudoalgebra and $M$ be a representation with $\psi = \cdot * \cdot : L \otimes M \rightarrow H^{\otimes 2} \otimes_H M$ being the action map. Note that $\psi$ can be realized as a Maurer-Cartan element in the differential graded Lie algebra $\mathfrak{g}$. Hence we may consider the new differential graded Lie algebra $\mathfrak{g}_\psi$ twisted by the Maurer-Cartan element $\psi$.

\begin{proposition}
    With the above notations, $\{ C^\bullet (L, M), \delta \}$ is a differential graded Lie subalgebra of $\mathfrak{g}_{{\psi}}$.
\end{proposition}

\begin{proof}
    For any $f \in C^m (L, M)$ and $g \in C^n (L,M)$, it follows from the definition of the bracket that $\llbracket f, g \rrbracket_{{\psi}} = \llbracket f, g \rrbracket = 0$. This shows that $C^{\bullet +1} (L,M)$ is a graded Lie subalgebra of $(C^{\bullet +1}_{>} (L\oplus M, M), \llbracket ~, ~ \rrbracket_{{\psi}})$. Next, for any $\theta \in C^n (L,M)$, we have
    \begin{align*}
        (d \theta) (x_1, \ldots, x_{n+1} ) =~& \llbracket \rho_L + \rho_M , \theta \rrbracket (x_1, \ldots, x_{n+1} ) \\
        =~& \sum_{i < j} (-1)^{i+j +1} (\sigma_{\substack{1 \rightarrow i \\
    2 \rightarrow j}} \otimes_H 1) \theta ([x_i * x_j]_L, x_1, \ldots, \widehat{x_i}, \ldots, \widehat{x_j}, \ldots, x_{n+1}).
    \end{align*}
    On the other hand,
    \begin{align*}
        \llbracket {\psi}, \theta \rrbracket (x_1, \ldots, x_{n+1}) =~& \big(  i_{{\psi}} \theta - (-1)^{n-1} ~i_\theta {\psi}  \big) (x_1, \ldots, x_{n+1}) \\
        =~& - (-1)^{n-1} \sum_{\sigma \in \mathrm{Sh} (n,1)} (-1)^\sigma (\sigma_{i \rightarrow \sigma (i)} \otimes_H 1) ~ {\psi} \big(  \theta (x_{\sigma (1)}, \ldots, x_{\sigma (n)}), x_{\sigma (n+1)} \big)  \\
        =~& \sum_{i=1}^{n+1} (-1)^i ~x_i * \theta (x_1, \ldots, \widehat{x_i}, \ldots, x_{n+1}).
    \end{align*}
    Thus, we have $d_{{\psi}} (\theta) = d (\theta) + \llbracket{\psi}, \theta \rrbracket = \delta (\theta) \in C^{n+1} (L,M)$. Here $\delta$ is the coboundary operator of the Lie $H$-pseudoalgebra $(L, [\cdot * \cdot ]_L)$ with coefficients in the representation $M$. This proves the result.
\end{proof}

The above proposition shows that $( C^{\bullet + 1} (L, M), \llbracket ~, ~\rrbracket_\psi, d_\psi)$ is a differential graded Lie algebra. Moreover, the cohomology groups induced by the differential $d_\psi$ coincide with the cohomology groups $H^\bullet (L, M)$ of the Lie $H$-pseudoalgebra $L$ with coefficients in the representation $M$.

\section{Inducibility of a pair of Lie {\em H}-pseudoalgebra automorphisms and the Wells map}\label{sec5}

Given a non-abelian extension of Lie $H$-pseudoalgebras, here we study the inducibility problem for a pair of Lie $H$-pseudoalgebra automorphisms. To find out the corresponding obstruction, we define the Wells map in the present context. Finally, we construct a short exact sequence connecting various automorphism groups and the Wells map.

Let $0 \rightarrow (M , [\cdot * \cdot]_M) \xrightarrow{i} (E, [\cdot * \cdot]_E) \xrightarrow{p} (L, [\cdot * \cdot]_L) \rightarrow 0$ be a given non-abelian extension of the Lie $H$-pseudoalgebra $L$ by the Lie $H$-pseudoalgebra $M$. We denote $\mathrm{Aut}_M (E)$ by the group of all Lie $H$-pseudoalgebra automorphisms $\gamma \in \mathrm{Aut}(E)$ for which $\gamma (M) \subset M$. For any $\gamma \in \mathrm{Aut}_M (E)$, we obviously have $\gamma|_M \in \mathrm{Aut} (M)$. Further, if $s$ is any section of the map $p$, then we can define a $H$-linear map $\overline{\gamma} : L \rightarrow L$ by $\overline{\gamma} (x) = p \gamma s (x)$, for $x \in L$. If $s'$ is any other section of the map $p$, then we have $(s- s') (x) \in \mathrm{ker }(p) = \mathrm{im }(i) \cong M$. Therefore, $\gamma (s- s')(x) \in M$ and hence $ p \gamma (s- s')(x) = 0$. This shows that the map $\overline{ \gamma}$ is independent of the choice of $s$. Next, we claim that the map $\overline{\gamma}: L \rightarrow L$ is a Lie $H$-pseudoalgebra automorphism of $L$. To see this, we first observe that the left $H$-module $L$ can be regarded as a left $H$-submodule of $E$ via any section $s$. The space $E$ is isomorphic to $M \oplus s(L)$. With this identification, the map $p$ is simply the projection onto $L$. Since $\gamma$ is bijective on $E$ preserving the $H$-module $M$, it must be bijective on the $H$-module $L$. In other words, the map $\overline{ \gamma} = p \gamma s$ is bijective on $L$. Further, for any $x, y \in L$, we have
\begin{align*}
    [ \overline{\gamma} (x) * \overline{\gamma} (y)]_L =~& [p \gamma s (x) * p \gamma s (y)]_L \\
    =~& (\mathrm{id}_{H^{\otimes 2}} \otimes_H p \gamma ) [s(x) * s(y)]_E \\
    =~& (\mathrm{id}_{H^{\otimes 2}} \otimes_H p \gamma ) \big(  [s(x) * s(y)]_E - \chi (x, y)  \big) ~~~~ (\text{as } \gamma (M) \subset M \text{ and } p|_M = 0)\\
    =~& (\mathrm{id}_{H^{\otimes 2}} \otimes_H p \gamma ) \big( (\mathrm{id}_{H^{\otimes 2}} \otimes_H s ) [x*y]_L   \big)
    \\
    =~& (\mathrm{id}_{H^{\otimes 2}} \otimes_{H} p \gamma s) [x * y]_L =  (\mathrm{id}_{H^{\otimes 2}} \otimes_{H} \overline{\gamma}) [x * y]_L.
\end{align*}
Hence $\overline{\gamma} \in \mathrm{Aut} (L)$ which proves our claim. Therefore, we obtain a group homomorphism 
\begin{align*}
    \tau : \mathrm{Aut}_M (E) \rightarrow \mathrm{Aut} (M) \times \mathrm{Aut}(L), ~ \tau (\gamma) = (\gamma|_M, \overline{\gamma}).
\end{align*}

\begin{definition}
    A pair $(\beta, \alpha) \in \mathrm{Aut} (M) \times \mathrm{Aut}(L)$ of Lie $H$-pseudoalgebra automorphisms is said to be {\em inducible} if there exists an automorphism $\gamma \in \mathrm{Aut}_M (E)$ such that $\tau (\gamma) = (\beta, \alpha)$, i.e., the pair $(\beta, \alpha)$ lies in the image of $\tau$.
\end{definition}

Let $s: L \rightarrow E$ be any section of the map $p$ and let $(\chi, \psi)$ be the non-abelian $2$-cocycle corresponding to the given non-abelian extension induced by $s$. Let $(\beta, \alpha) \in \mathrm{Aut} (M) \times \mathrm{Aut} (L)$ be any pair of Lie $H$-pseudoalgebra automorphisms. We define a skew-symmetric map $\chi_{(\beta, \alpha)} \in \mathrm{Hom}_{H^{\otimes 2}} (L \otimes L, H^{\otimes 2} \otimes_H M)$ and a map $\psi_{(\beta, \alpha)} \in \mathrm{Hom}_{H^{\otimes 2}} (L \otimes M, H^{\otimes 2} \otimes_H M)$ by
\begin{align*}
\chi_{(\beta, \alpha)} (x, y) := (\mathrm{id}_{H^{\otimes 2}} \otimes_H \beta) \chi (\alpha^{-1} (x) , \alpha^{-1} (y)) ~ \text{ and } ~ \psi_{(\beta, \alpha)} (x, u) := (\mathrm{id}_{H^{\otimes 2}} \otimes_H \beta) \psi (\alpha^{-1} (x) , \beta^{-1} (y)),
\end{align*}
for $x, y \in L$ and $u \in M$. Then we have the following.

\begin{lemma}
The pair $( \chi_{(\beta, \alpha)}, \psi_{(\beta, \alpha)}  )$ is a non-abelian $2$-cocycle of $(L, [\cdot * \cdot]_L)$ with values in $(M, [\cdot * \cdot]_M).$
\end{lemma}

\begin{proof}
    Since $(\chi, \psi)$ is a non-abelian $2$-cocycle, we have the identities (\ref{deri-iden}), (\ref{first-iden}) and (\ref{second-iden}). In these identities, if we replace $x, y, z, u, v$ by $\alpha^{-1} (x) , \alpha^{-1} (y), \alpha^{-1} (z), \beta^{-1} (u), \beta^{-1}(v)$, we get the cocycle conditions for $( \chi_{(\beta, \alpha)}, \psi_{(\beta, \alpha)}  )$. For example, if we replace $x, u, v$ by $\alpha^{-1} (x), \beta^{-1}(u), \beta^{-1}(v)$ in the identity (\ref{deri-iden}), we obtain
    \begin{align*}
        [\psi (\alpha^{-1} (x), \beta^{-1}(u)) * \beta^{-1}(v)]_M = \psi (\alpha^{-1}(x), [\beta^{-1} (u) * \beta^{-1}(v)]_M) - [\beta^{-1}(u) * \psi (\alpha^{-1} (x) , \beta^{-1} (v))]_M
    \end{align*}
    which can be written as
    \begin{align*}
        [  (\mathrm{id}_{H^{\otimes 2}} \otimes_H  \beta^{-1}) \psi_{(\beta, \alpha)} (x, u) * \beta^{-1} (v)  ]_M =~& \psi (\alpha^{-1} (x) , (\mathrm{id}_{H^{\otimes 2}} \otimes_H  \beta^{-1}) [u * v]_M ) \\
        &- [\beta^{-1} (u) *  (\mathrm{id}_{H^{\otimes 2}} \otimes_H  \beta^{-1}) \psi_{(\beta, \alpha)} (x, v)].
    \end{align*}
    If we apply $(\mathrm{id}_{H^{\otimes 3}} \otimes_H \beta)$ to both sides, we simply get
    \begin{align*}
        [\psi_{(\beta, \alpha)} (x, u) * v]_M  = \psi_{(\beta, \alpha)} (x, [u * v]_M) - [u * \psi_{(\beta, \alpha)} (x, v)]_M.
    \end{align*}
    This is nothing but the identity (\ref{deri-iden}) for $\psi_{(\beta, \alpha)}$. Similarly, we will obtain the identities (\ref{first-iden}) and (\ref{second-iden}) for the pair $(\chi_{(\beta, \alpha)}, \psi_{(\beta, \alpha)})$. Hence $(\chi_{(\beta, \alpha)}, \psi_{(\beta, \alpha)})$ is a non-abelian $2$-cocycle.
\end{proof}

Note that the non-abelian $2$-cocycle  $( \chi_{(\beta, \alpha)}, \psi_{(\beta, \alpha)}  )$ depends on the section $s$ as the non-abelian $2$-cocycle $( \chi, \psi )$ is so. We now define a map $\mathcal{W} : \mathrm{Aut} (M) \times \mathrm{Aut}(L) \rightarrow H^2_{nab} (L, M)$ by
\begin{align*}
    \mathcal{W} \big(   (\beta, \alpha) \big) := [ ( \chi_{(\beta, \alpha)}, \psi_{(\beta, \alpha)}  ) -  ( \chi, \psi ) ],
\end{align*}
the equivalence class of $( \chi_{(\beta, \alpha)}, \psi_{(\beta, \alpha)}  ) -  ( \chi, \psi )$.
The map $\mathcal{W}$ is called the {\em Wells map}. The first important property of the Wells map is given by the following.

\begin{lemma}
    The Wells map does not depend on the chosen section.
\end{lemma}

\begin{proof}
    Let $s' : L \rightarrow E$ be any other section of the map $p$ and let $(\chi', \psi')$ be the non-abelian $2$-cocycle corresponding to the given non-abelian extension and induced by the section $s'$. We have already seen in Section \ref{sec3} that the non-abelian $2$-cocycles $(\chi, \psi)$ and $(\chi', \psi')$ are equivalent, and an equivalence is given by the map $\varphi := s - s'$. 

    Next, we observe that
    \begin{align*}
        \psi_{(\beta, \alpha)} (x, u) - \psi'_{(\beta, \alpha)} (x, u) =~& (\mathrm{id}_{H^{\otimes 2}} \otimes_H \beta) \big\{  \psi ( \alpha^{-1} (x), \beta^{-1} (u)) -  \psi' ( \alpha^{-1} (x), \beta^{-1} (u)) \big\} \\
        =~& (\mathrm{id}_{H^{\otimes 2}} \otimes_H \beta) [\varphi \alpha^{-1} (x) * \beta^{-1} (u)]_M = [\beta \varphi \alpha^{-1} (x) * u]_M,
    \end{align*}
    for $x \in L$, $u \in M$. Similarly, for $x, y \in L$,
    \begin{align*}
        \chi_{(\beta, \alpha)} (x, y) - \chi'_{(\beta, \alpha)} (x, y) =~& \psi'_{(\beta, \alpha)} (x, \beta \varphi \alpha^{-1} (y)) - (\sigma_{12} \otimes_H 1)~ \psi'_{(\beta, \alpha)} (y, \beta \varphi \alpha^{-1} (x)) \\
        &- (\beta \varphi \alpha^{-1}) [x * y]_L + [\beta \varphi \alpha^{-1} (x) * \beta \varphi \alpha^{-1} (y)]_M.
    \end{align*}
    This proves that the non-abelian $2$-cocycles $(\chi_{(\beta, \alpha)} , \psi_{(\beta, \alpha)})$ and $(\chi'_{(\beta, \alpha)} , \psi'_{(\beta, \alpha)})$ are equivalent. Hence the $2$-cocycles 
    \begin{align*}
        (\chi_{(\beta, \alpha)} , \psi_{(\beta, \alpha)}) - (\chi , \psi) \quad \text{ and } \quad (\chi'_{(\beta, \alpha)} , \psi'_{(\beta, \alpha)}) - (\chi' , \psi')
    \end{align*}
    are equivalent, and equivalence is given by the map $\beta \varphi \alpha^{-1} - \varphi$. This proves the desired result.
\end{proof}

\begin{remark}
    Note that the group $\mathrm{Aut} (M) \times \mathrm{Aut} (L)$ acts on the space $H^2_{nab} (L, M)$ by
    \begin{align*}
        (\beta, \alpha) \cdot [(\chi, \psi)] = [(\chi_{(\beta, \alpha)} , \psi_{(\beta, \alpha)})],
    \end{align*}
    for any $(\beta, \alpha) \in \mathrm{Aut} (M) \times \mathrm{Aut} (L)$ and $[(\chi, \psi)] \in H^2_{nab} (L, M)$. With this notation, the Wells map $\mathcal{W}$ is simply given by 
    \begin{align*}
         \mathcal{W} ((\beta, \alpha )) = (\beta, \alpha) \cdot [(\chi, \psi)] - [(\chi, \psi)],
    \end{align*}
    where $[(\chi, \psi)] \in H^2_{nab} (L, M)$ is the non-abelian cohomology class corresponding to the given non-abelian extension. This shows that the Wells map can be seen as a principal crossed homomorphism in the group cohomology complex of $\mathrm{Aut} (M) \times \mathrm{Aut} (L)$ with values in $H^2_{nab} (L, M)$.
\end{remark}

We are now in a position to find a necessary and sufficient condition for a pair of Lie $H$-pseudoalgebra automorphisms to be inducible. 

\begin{theorem}\label{nece-suff}(Necessary and sufficient condition) Let $ 0 \rightarrow (M, [\cdot * \cdot]_M) \xrightarrow{i} (E, [\cdot * \cdot]_E) \xrightarrow{p} (L, [\cdot * \cdot]_L) \rightarrow 0$  be a non-abelian extension of Lie $H$-pseudoalgebras. A pair $(\beta, \alpha) \in \mathrm{Aut}(M) \times \mathrm{Aut}(L)$ of Lie $H$-pseudoalgebra automorphisms is inducible if and only if $\mathcal{W} \big( (\beta, \alpha) \big) = 0$.
\end{theorem}

\begin{proof}
    Let the pair $(\beta, \alpha) \in \mathrm{Aut} (M) \times \mathrm{Aut}(L)$ be inducible. In other words, there exists a Lie $H$-pseudoalgebra automorphism $\gamma \in \mathrm{Aut}_M (E)$ such that $\gamma|_M  = \beta$ and $\overline{\gamma} := p \gamma s = \alpha$ (for any section $s$ of the map $p$). We define a map $\varphi : L \rightarrow E$ by $\varphi (x) := (\gamma s - s \alpha) \alpha^{-1} (x) = (\gamma s \alpha^{-1} - s)(x)$, for $x \in L$. Observe that $p \varphi (x) = (p \gamma s \alpha^{-1} - ps)(x) = 0$, for $x \in L$. This shows that $\mathrm{im} (\varphi) \subset \mathrm{ker} (p) = \mathrm{im} (i) \cong M$. Thus, $\varphi$ can be realized as a map $\varphi : L \rightarrow M$. Next, if $(\chi, \psi)$ is the non-abelian $2$-cocycle corresponding to the given non-abelian extension and induced by the section $s$, then we have
    \begin{align*}
        \psi_{(\beta, \alpha)} (x, u) - \psi (x, u) =~& (\mathrm{id}_{H^{\otimes 2}} \otimes_H \beta) \psi (\alpha^{-1} (x) , \beta^{-1} (u)) - \psi (x, u) \\
        =~& (\mathrm{id}_{H^{\otimes 2}} \otimes_H \beta) [s\alpha^{-1} (x) * \beta^{-1} (u)]_E - [s(x) * u]_E \\
        =~& [\gamma s \alpha^{-1} (x) * u]_E - [s(x) * u]_E   \quad (\text{as } \beta = \gamma|_M) \\
        =~& [\varphi (x) * u]_E.
    \end{align*}
    Similarly, by a direct calculation, we get that
    \begin{align*}
        \chi_{(\beta, \alpha)} (x, y) - \chi (x, y) = \psi (x, \varphi (y) ) - (\sigma_{12} \otimes_H 1) ~\psi (y, \varphi (x)) - (\mathrm{id}_{H^{\otimes 2}} \otimes_H  \varphi) [x*y]_L + [\varphi (x) * \varphi (y)]_M, 
    \end{align*}
    for $x, y \in L.$ These two observations shows that the non-abelian $2$-cocycles $(\chi_{(\beta, \alpha)}, \psi_{(\beta, \alpha)})$ and $(\chi, \psi)$ are equivalent by the map $\varphi$. Hence $\mathcal{W} ((\beta, \alpha)) = [ (\chi_{(\beta, \alpha)}, \psi_{(\beta, \alpha)}) - (\chi, \psi)  ] = 0$.

    Conversely, let $(\beta, \alpha) \in \mathrm{Aut}(M) \times \mathrm{Aut} (L)$ be such that $\mathcal{W} ((\beta, \alpha )) = 0$. Consider any section $s$ of the map $p$ and let $(\chi, \psi)$ be the non-abelian $2$-cocycle corresponding to the given non-abelian extension and induced by $s$. Since $\mathcal{W}((\beta, \alpha)) = 0$, it follows that the non-abelian $2$-cocycles $(\chi_{(\beta, \alpha)}, \psi_{(\beta, \alpha)}) $ and $(\chi, \psi) $ are equivalent, and say by the map $\varphi : L \rightarrow M$. We now define a $H$-linear map $\gamma : E \rightarrow E$ by
    \begin{align*}
        \gamma (e) = \gamma (u + s(x)) = (\beta (u) + \varphi \alpha (x)) + s (\alpha (x)), \text{ for } e = u+ s(x) \in E.
    \end{align*}
    Here we have used the fact that any element $e \in E$ can be written as $e= u +s (x)$, for some $u \in M$ and $x\in L$.

    We claim that $\gamma$ is bijective. Suppose $\gamma (e) = \gamma (u+ s(x)) = 0$. Then it follows that $s (\alpha (x)) = 0$. Since $s$ and $\alpha$ are both injective maps, we have $x = 0$. Using this in the definition of $\gamma$, we get that $\beta (u) = 0$ which implies that $u = 0$. That is, $e = u + s(x) = 0$. Hence $\gamma$ is injective. To show that $\gamma$ is surjective, we consider an arbitrary element $e = u + s(x) \in E$. Then the element $e' = (\beta^{-1} (u) - \beta^{-1} \varphi (x)) + s (\alpha^{-1}(x)) \in E$ and we have $\gamma (e') = e$. This proves that $\gamma$ is surjective and hence our claim follows.

    In the following, we show that the map $\gamma: E \rightarrow E$ is a Lie $H$-pseudoalgebra homomorphism. Take any two elements $e = u + s(x)$ and $e' = v+ s(y)$ from $E$. Then we have
    \begin{align*}
        &[\gamma (e) * \gamma (e')]_E \\
        &= [  (\beta (u) + \varphi \alpha (x) + s \alpha (x)) * (\beta (v) + \varphi \alpha (y) + s \alpha (y))]_E\\
        &= [\beta (u) * \beta (v)]_E + [\beta (u) * \varphi \alpha (y)]_E + [\beta(u) * s \alpha (y)]_E + [\varphi \alpha (x) * \beta (v)]_E + [\varphi \alpha (x) * \varphi \alpha (y)]_E \\
        & \quad  + [\varphi \alpha (x) * s \alpha (y)]_E + [s \alpha (x) * \beta (v) ]_E + [s \alpha (x) * \varphi \alpha (y)]_E + [s \alpha (x) * s \alpha (y)]_E\\
        &= [\beta (u) * \beta(v)]_E - (\sigma_{12} \otimes_H \beta) \psi (y, u) + (\sigma_{12} \otimes_H \beta) \psi (\alpha (y), \beta (u)) \\
        & \quad  -  (\sigma_{12} \otimes_H \beta) \psi (\alpha (y), \beta (u)) +  (\mathrm{id}_{H^{\otimes 2}} \otimes_H \beta) \psi (x, v) - \psi (\alpha (x), \beta (v)) \\
        & \quad  +  (\mathrm{id}_{H^{\otimes 2}} \otimes_H \beta) \big(  \chi (x, y) - \chi (\alpha(x), \alpha (y)) - \psi (  \alpha(x), \varphi \alpha (y)) +  (\sigma_{12} \otimes_H 1) \psi (\alpha (y), \varphi \alpha (x)) + [x * y]_L  \big) \\
        & \quad  -  (\sigma_{12} \otimes_H \beta) \psi (\alpha (y), \varphi \alpha (x)) + \psi (\alpha (x), \beta (v)) + \psi (\alpha (x), \varphi \alpha (y)) \\
        & \quad  +  (\mathrm{id}_{H^{\otimes 2}} \otimes_H \beta) \chi (\alpha (x), \alpha (y)) +  (\mathrm{id}_{H^{\otimes 2}} \otimes_H s) [\alpha (x) * \alpha (y)]_L \\
        &= (\mathrm{id}_{H^{\otimes 2}} \otimes_H \beta) \big(  [u * v]_M - (\sigma_{12} \otimes_H 1) \psi(y, u) + \psi (x, v) + \chi (x, y)  \big) \\
        & \quad  + (\mathrm{id}_{H^{\otimes 2}} \otimes_H \lambda) [x * y]_L + (\mathrm{id}_{H^{\otimes 2}} \otimes_H s) [\alpha (x) * \alpha (y)]_L \\
        &= (\mathrm{id}_{H^{\otimes 2}} \otimes_H \beta) \big(  [u * v]_M + [u * s(y)]_E + [s(x) * v]_E + \chi (x, y)  \big)  + (\mathrm{id}_{H^{\otimes 2}} \otimes_H \lambda) [x * y]_L \\
        & \quad  + (\mathrm{id}_{H^{\otimes 2}} \otimes_H s \alpha) [x * y]_L \\
        &= (\mathrm{id}_{H^{\otimes 2}} \otimes_H \gamma) \big(  [u * v]_E + [u * s(y)]_E + [s(x) * v]_E + \chi (x, y)  +  (\mathrm{id}_{H^{\otimes 2}} \otimes_H s) [x * y]_L \big)\\
        &= (\mathrm{id}_{H^{\otimes 2}} \otimes_H \gamma) \big(   [u * v]_E + [u * s(y)]_E + [s(x) * v]_E + [s(x) * s (y)]_E \big)  \\
        &= (\mathrm{id}_{H^{\otimes 2}} \otimes_H \gamma) [ (u + s(x)) * (v + s(y))   ]_E\\
        &= (\mathrm{id}_{H^{\otimes 2}} \otimes_H \gamma) [e * e']_E.
    \end{align*}
    Combining this with the fact that $\gamma$ is bijective, we have that $\gamma \in \mathrm{Aut}(E)$. Finally, for $u \in M$ and $x \in L$, 
    \begin{align*}
        &\gamma (u) = \gamma (u + s(0)) = \beta (u),\\
        &\overline{\gamma} (x) = p \gamma s (x) = p \gamma (0 + s(x)) = p (\varphi \alpha (x) + s (\alpha (x))) = ps (\alpha (x) )= \alpha (x). 
    \end{align*}
    Hence $\tau = (\gamma|_M, \overline{\gamma}) =(\beta, \alpha)$ which proves that the pair $(\beta, \alpha)$ is inducible.
\end{proof}

\begin{remark}
    It follows from Theorem \ref{nece-suff} that $\mathcal{W} \big( (\beta, \alpha) \big)$ is an obstruction for the inducibility of the pair $(\beta, \alpha)$ of Lie $H$-pseudoalgebra automorphisms. Thus, if the Wells map vanishes identically, then any pair of Lie $H$-pseudoalgebra automorphisms in $\mathrm{Aut} (M) \times \mathrm{Aut}(L)$ is inducible.
\end{remark}

In the context of classical Lie algebras, while considering the abelian extensions, it is well-known that the Wells map fits into a short exact sequence \cite{bardakov}. Our next aim is to generalize this sequence in the context of Lie $H$-pseudoalgebras (and for non-abelian extensions).

Let $ 0 \rightarrow (M, [\cdot * \cdot]_M) \xrightarrow{i} (E, [\cdot * \cdot]_E) \xrightarrow{p} (L, [\cdot * \cdot]_L) \rightarrow 0$ be a given non-abelian extension of Lie $H$-pseudoalgebras. We define a subgroup $\mathrm{Aut}_M^{M, L} (E) \subset \mathrm{Aut}_M (E)$ by
\begin{align*}
  \mathrm{Aut}_M^{M, L} (E) = \{ \gamma \in  \mathrm{Aut}_M (E) | ~ \tau (\gamma) = (\mathrm{id}_M, \mathrm{id}_L) \}. 
\end{align*}
That is, $\mathrm{Aut}_M^{M, L} (E)$ is the group of all Lie $H$-pseudoalgebra automorphisms on $E$ whose restriction on $M$ is the identity map $\mathrm{id}_M$ and projection on $L$ is also the identity map $\mathrm{id}_L$. Then we have the following result.

\begin{theorem}\label{wells-thm} (Wells exact sequence) Let $ 0 \rightarrow (M, [\cdot * \cdot]_M) \xrightarrow{i} (E, [\cdot * \cdot]_E) \xrightarrow{p} (L, [\cdot * \cdot]_L) \rightarrow 0$ be a non-abelian extension of Lie $H$-pseudoalgebras. Then there is an exact sequence connecting various automorphism groups
\begin{align}\label{wells-es}
    \xymatrix{
    1 \ar[r] & \mathrm{Aut}_M^{M, L} (E) \ar[r]^\iota & \mathrm{Aut}_M (E) \ar[r]^\tau & \mathrm{Aut} (M) \times \mathrm{Aut}(L) \ar[r]^{\mathcal{W}} & H^2_{nab} (L, M).
    }
\end{align}
\end{theorem}

\begin{proof}
    Note that the sequence (\ref{wells-es}) is exact at the first term as the inclusion map $\iota : \mathrm{Aut}_M^{M, L} (E) \rightarrow \mathrm{Aut}_M (E)$ is injective. To show that the sequence is exact at the second term, we first take an element $\gamma \in \mathrm{ker} (\tau)$. Then we have $\tau (\gamma) = (\gamma|_M, \overline{\gamma}) = (\mathrm{id}_M, \mathrm{id}_L)$. Hence $\gamma \in \mathrm{Aut}_{M}^{M, L} (E)$. Conversely, if $\gamma \in \mathrm{Aut}_{M}^{M, L} (E)$ then we have $\gamma \in \mathrm{ker} (\tau)$. Thus, $\mathrm{ker} (\tau) = \mathrm{Aut}_M^{M, L}(E) = \mathrm{im} (\iota)$. Hence the sequence is exact at the second term. Finally, we take an element $(\beta, \alpha) \in \mathrm{ker} (\mathcal{W})$. Then by Theorem \ref{nece-suff}, it follows that the pair $(\beta, \alpha)$ is inducible. Hence there exists a Lie $H$-pseudoalgebra automorphism $\gamma \in \mathrm{Aut}_M (E)$ such that $\tau (\gamma) = (\beta, \alpha)$, i.e. $(\beta, \alpha) \in \mathrm{im} (\tau)$. Conversely, if $(\beta, \alpha) \in \mathrm{im} (\tau)$ then $(\beta, \alpha)$ is inducible and hence $\mathcal{W}((\beta, \alpha)) = 0$ which shows that $(\beta, \alpha) \in \mathrm{ker} (\mathcal{W})$. Therefore, $\mathrm{ker} (\mathcal{W}) = \mathrm{im} (\tau)$ and thus the sequence is exact at the third place.
\end{proof}

In the following, we will discuss the Wells short exact sequence in abelian extensions of a Lie $H$-pseudoalgebra. Let
 \begin{align}\label{abb-ext}
     \xymatrix{
     0 \ar[r] & (M, [\cdot * \cdot]_M = 0) \ar[r]^i & (E, [\cdot * \cdot]_E) \ar[r]^p & (L, [\cdot * \cdot]_L) \ar[r] & 0
     }
     \end{align}
     be an abelian extension of the Lie $H$-pseudoalgebra $(L, [\cdot * \cdot]_L)$ by the representation $M$. Let $\psi \in \mathrm{Hom}_{H^{\otimes 2}} (L \otimes M, H^{\otimes 2} \otimes_H M)$ denotes the action map of the representation. We define a subgroup $C_\psi \subset \mathrm{Aut} (M) \times \mathrm{Aut} (L)$ by
     \begin{align*}
         C_\psi : = \{ (\beta, \alpha) \in \mathrm{Aut} (M) \times \mathrm{Aut} (L) |~ (\mathrm{id}_{H^{\otimes 2}} \otimes_H \beta) \psi (x, u) = \psi (\alpha (x), \beta(u)), \text{for } x \in L, u \in M \}.
     \end{align*}
We observe that if $\gamma \in \mathrm{Aut}_M (E)$ then $\tau (\gamma) = (\gamma|_M, \overline{\gamma}) \in C_\psi$. 

Suppose the skew-symmetric map $\chi \in \mathrm{Hom}_{H^{\otimes 2}} (L \otimes L, H^{\otimes 2} \otimes_H M)$ is the $2$-cocycle corresponding to the given abelian extension. Then for any pair $(\beta, \alpha) \in \mathrm{Aut} (M) \times \mathrm{Aut} (L)$, the map $\chi_{(\beta, \alpha)}$ need not be a $2$-cocycle. However, if $(\beta, \alpha) \in C_\psi$ then $\chi_{(\beta, \alpha)}$ turns out to be a $2$-cocycle. Thus, the Wells map here is given by $\mathcal{W} : C_\psi \rightarrow H^2(L, M)$ with $\mathcal{W} ((\beta, \alpha )) := [\chi_{(\beta, \alpha)} - \chi]$, for $(\beta, \alpha)\in C_\psi$. With the above notations, we have the following results which are similar to Theorem \ref{nece-suff} and Theorem \ref{wells-thm}, respectively.

\begin{theorem}
    Given an abelian extension (\ref{abb-ext}), a pair $(\beta, \alpha) \in \mathrm{Aut} (M) \times \mathrm{Aut} (L)$ is inducible if and only if $(\beta, \alpha) \in C_\psi$ and $\mathcal{W} ((\beta, \alpha)) = 0$.
\end{theorem}

\begin{theorem}
    Given an abelian extension (\ref{abb-ext}), we have the following short exact sequence
    \begin{align*}
    \xymatrix{
        1 \ar[r] & \mathrm{Aut}_M^{M,L} (E) \ar[r]^{\iota} & \mathrm{Aut}_M (E) \ar[r]^{\tau} & C_\psi \ar[r]^{\mathcal{W}} & H^2(L,M).
        }
    \end{align*}
\end{theorem}


\medskip

\noindent {\bf Acknowledgements.} The author would like to thank the Department of Mathematics, IIT Kharagpur for providing the beautiful academic atmosphere where the research has been carried out.

\medskip

\noindent {\bf Data Availability Statement.} Data sharing does not apply to this article as no new data were created or analyzed in this study.

\end{document}